\documentclass[12pt,a4paper]{amsart}
\usepackage{mathrsfs}
\usepackage{amsfonts}
\usepackage[active]{srcltx}
\usepackage[notref,notcite]{showkeys}
\usepackage{enumerate}

\usepackage{amsmath,amssymb,xspace,amsthm}
\newtheorem{theorem}{Theorem}
\newtheorem{lemma}[theorem]{Lemma}
\newtheorem{proposition}[theorem]{Proposition}
\numberwithin{equation}{section}

\def\span{\operatorname{span}}

\newcommand{\Vir}{\mathrm{Vir}}

\newcommand{\C}{\ensuremath{\mathbb C}\xspace}
\newcommand{\Q}{\ensuremath{\mathbb{Q}}\xspace}
\renewcommand{\a}{\ensuremath{\alpha}}
\renewcommand{\b}{\ensuremath{\beta}}

\newcommand{\Z}{\ensuremath{\mathbb{Z}}\xspace}
\newcommand{\N}{\ensuremath{\mathbb{N}}\xspace}

\newcommand{\V}{\ensuremath{\mathfrak{V}}\xspace}

\newcommand{\R}{\ensuremath{\mathbb{R}}\xspace}

\renewcommand{\phi}{\varphi}

\renewcommand{\geq}{\geqslant}

\def\LL{\mathcal{L}}

\begin{document}
\title[Weight modules]{A  class of simple weight Virasoro modules}
\author{Genqiang Liu, Rencai Lu,   Kaiming Zhao}
\date{Sept.26, 2012}
\maketitle

\begin{abstract} For a simple module $M$ over the positive part of the Virasoro algebra
(actually for any simple module over some finite dimensional
solvable Lie algebras $\mathfrak{a}_r$)  and any   $\alpha\in\C$, a
class of weight modules $\mathcal {N}(M, \alpha)$ over the Virasoro
algebra are constructed. 
The necessary and sufficient condition for $\mathcal {N}(M, \a)$ to
be simple is obtained. We also determine the necessary and
sufficient conditions for two such irreducible Virasoro modules to
be isomorphic. Many examples for such irreducible Virasoro modules
with different features are provided. In particular the irreducible
weight Virasoro modules $\Gamma(\alpha_1, \a_2,\lambda_1,
\lambda_2)$ are defined on the polynomial algebra $\C[x]\otimes
\C[t,t^{-1}]$ for any $\alpha_1, \a_2, \lambda_1\lambda_2\in\C$ with
$\lambda_1$ or $\lambda_2$ nonzero. By   twisting the weight modules
$\mathcal {N}(M, \a)$ we also obtain nonweight simple Virasoro
modules $\mathcal {N}(M, \b)$ for any $\b\in\C[t,t^{-1}]$.
\end{abstract}


\vskip 5pt
\noindent
{\em 2000  Math. Subj. Class.:}
17B10, 17B20, 17B65, 17B66, 17B68

\vskip 10pt

\section{Introduction}
We denote by $\mathbb{Z}$, $\mathbb{Z}_+$, $\N$, $\R$ and
$\mathbb{C}$ the sets of  all integers, nonnegative integers,
positive integers, real numbers and complex numbers, respectively.
For a Lie algebra $\mathfrak{a}$ we denote by $U(\mathfrak{a})$ the
universal enveloping algebra of $\mathfrak{a}$.

Let $\mathfrak{V}$ denote the complex {\em Virasoro algebra}, that
is the Lie algebra with basis $\{\mathtt{c}, d_i:i\in\mathbb{Z}\}$
and the Lie bracket defined (for $i,j\in\mathbb{Z}$) as follows:
\begin{displaymath}
[d_i,d_j]=(j-i) d_{i+j}+\delta_{i,-j}\frac{i^3-i}{12}{c};\quad
[d_i,{c}]=0.
\end{displaymath}
The algebra $\mathfrak{V}$ is one of the most   important Lie
algebras both in mathematics and in mathematical physics, see for
example \cite{KR,IK} and references therein.
 The Virasoro algebra theory has been widely used in many physics areas
and other mathematical branches, for example, quantum physics [GO],
conformal field theory [FMS], Higher-dimensional WZW models [IKUX,
IUK], Kac-Moody algebras [K, MoP], vertex algebras [LL], and so on.

The representation theory on the Virasoro algebra has  attracted a
lot of attentions from mathematicians and physicists. There are two
classical families of simple Harish-Chandra $\mathfrak{V}$-modules:
highest weight modules (completely described in \cite{FF}) and the
so-called intermediate series modules. In \cite{Mt} it is shown that
these two families exhaust all simple weight Harish-Chandra modules.
In \cite{MZ1} it is even shown that the above modules exhaust all
simple weight modules admitting a nonzero finite dimensional weight
space. Very naturally, the next important task is to study simple
weight modules with infinite dimensional weight spaces. The first
such examples were constructed by taking  the tensor product of some
highest weight modules and some intermediate series modules in
\cite{Zh} in 1997 , and then Conley and Martin gave another class of
such examples with four parameters in \cite{CM} in 2001. Since then
(for the last decade) there were no other such irreducible Virasoro
modules found. The present paper is to construct a huge class of
such irreducible Virasoro modules.

At the same time for the last decade, various families of nonweight
simple Virasoro modules were studied in
\cite{OW,LGZ,LZ,FJK,Ya,GLZ,MW,OW2}. These (except the Weyl modules
in \cite{LZ}) are basically various versions of Whittaker modules
constructed using different tricks. In particular, all the above
Whittaker modules (but the ones in \cite{MW})and even more were
described in a uniform way in \cite{MZ3}. Now we briefly describe
the  main results in the present paper.

\

Denote by $\mathfrak{V}_+$ the Lie subalgebra of $\mathfrak{V}$
spanned by all $d_i$ with $i\geq 0$. For $r\in\mathbb{Z}_+$, denote
by $\mathfrak{V}^{(r)}_+$ the Lie subalgebra of $\mathfrak{V}$
generated by all $d_i$, $i>r$, and by $ \mathfrak{a}_r$ the quotient
algebra $\mathfrak{V}_+/\mathfrak{V}^{(r)}_+$. By $\bar d_i$ we
denote the image of $d_i$ in $ \mathfrak{a}_r$. A classification for
all irreducible modules over $ \mathfrak{a}_1$ was given in
\cite{Bl}, while  a classification for all irreducible modules over
$ \mathfrak{a}_2$ was recently obtained in \cite{MZ3}. The problem
is open for all other Lie algebras  $ \mathfrak{a}_r$ for $r\ge 2$.

\

The paper is organized as follows. In Sect.2, for any module $M$
over $\mathfrak{a}_r$ and  any $\alpha\in\C$, we define a weight
Virasoro module structure on the vector space $\mathcal {N}(M,
\a)=V\otimes \C[t^{\pm 1}]$. In Sect.3 we prove that the Virasoro
module $\mathcal {N}(M, \a)$ is simple if $r\ge 1$ and $M$ is an
infinite dimensional irreducible $\mathfrak{a}_r$-module. Thus we
obtain a huge class of irreducible weight Virasoro modules. We also
determine the necessary and sufficient conditions for two such
irreducible Virasoro modules to be isomorphic. In Sect.4 we compare
these irreducible weight modules $\mathcal {N}(M, \a)$ with other
known irreducible weight  Virasoro modules (\cite{CM, Zh}) to prove
that they are new. In Sect.5 we give concrete examples for such
irreducible Virasoro modules $\mathcal {N}(M, \a)$. In particular
the irreducible weight Virasoro modules $\Gamma(\alpha_1,
\a_2,\lambda_1, \lambda_2)$ are defined on the polynomial algebra
$\C[x]\otimes \C[t,t^{-1}]$ for any $\alpha_1, \a_2,
\lambda_1\lambda_2\in\C$ with $\lambda_1$ or $\lambda_2$ nonzero. In
Sect.6, by   twisting the weight modules $\mathcal {N}(M, \a)$ we
also obtain nonweight simple Virasoro modules $\mathcal {N}(M, \b)$
for any $\b\in\C[t,t^{-1}]$. These nonweight irreducible Virasoro
modules are also new.

\section{Constructing new weight Virasoro modules}

Let us start with a module $M$ over $\mathfrak{a}_r$ and $\alpha\in
\C$. We define a Virasoro module structure on the vector space
$\mathcal {L}(M, \a)=V\otimes \C[t^{\pm 1}]$ as follows
\begin{equation}\label{def} d_m (v\otimes t^n)=((\a+n+\sum_{i=0}^r(\frac{m^{i+1}}{(i+1)!}\bar d_i))v)\otimes t^{n+m},\end{equation}
\begin{equation}c(v\otimes t^n)=0,\end{equation}
where $m, n\in\Z, v\in M$. This will be verified after the next
example.

\

We will always write $v(n)=v\otimes t^n$ for short.

\

 \noindent{\bf{Example 1.}} When $r=0$, any irreducible module $M=\C v$ over
the Lie algebra $\mathfrak{a}_0$ is one dimensional and given by  a
scalar $b\in\C$ with the action $\bar d_0v=bv$. In this case,
$\mathcal {N}(M, \a)=\C v\otimes \C[t^{\pm 1}]$ and (2.1) becomes
\begin{equation} d_m v(n)=(\a+n+bm)v(n+m),\end{equation}
 which is exactly the module of intermediate series (see \cite{KR}).

\begin{proposition} The actions (2.1) and (2.2) make $\mathcal {N}(M, \a)$ into  a module over $\mathfrak{V}$.
\end{proposition}

\begin{proof} We need to verify that $[d_m,d_k]v(n)=(d_md_k-d_kd_m)v(n)$ for all $m,n\in\Z$ and $v\in M$. By (2.1), we have
$$[d_m,d_k]v(n)=(k-m)(\a+n+\sum_{i=0}^r\frac{(m+k)^{i+1}}{(i+1)!}\bar d_i)v(n+m+k).$$
At the same time
$$\aligned &\ (d_md_k- d_kd_m)v(n)\\
=&\ \Big(\a+n+k+\sum_{i=0}^r\frac{m^{i+1}}{(i+1)!}\bar d_i\Big)
\Big(\a+n+\sum_{j=0}^r\frac{k^{j+1}}{(j+1)!}\bar d_j\Big)v(n+m+k)\\
-& \Big(\a+n+m+\sum_{i=0}^r\frac{k^{j+1}}{(j+1)!}\bar d_j\Big)
\Big(\a+n+\sum_{i=0}^r\frac{m^{i+1}}{(i+1)!}\bar d_i\Big)v(n+m+k)\\
=&(k-m)(\a+n)v(n+m+k)\\ &+\Big(k\sum_{i=0}^r\frac{k^{i+1}}{(i+1)!}
-m\sum_{i=0}^r\frac{m^{i+1}}{(i+1)!}\Big)\bar d_iv(n+m+k)\\
&\ +\sum_{i,j=0}^r\frac{m^{i+1}k^{j+1}}{(i+1)!(j+1)!}(j-i)\bar
d_{i+j}v(n+m+k)\endaligned$$
$$\aligned =&(k-m)(\a+n)v(n+m+k)\\
&+\Big(k\sum_{i=0}^r\frac{k^{i+1}}{(i+1)!}
-m\sum_{i=0}^r\frac{m^{i+1}}{(i+1)!}\Big)\bar d_iv(n+m+k)\\
&\ +\left(k\sum_{i,j=0}^r\frac{m^{i+1}k^j}{(i+1)!j!}\bar d_{i+j}
-m\sum_{i,j=0}^r\frac{m^ik^{j+1}}{i!(j+1)!}\bar
d_{i+j}\right)v(n+m+k)\\
=&(k-m)(\a+n)v(n+m+k)\\ &+\Big(k\sum_{i=0}^r\frac{k^{i+1}}{(i+1)!}
-m\sum_{i=0}^r\frac{m^{i+1}}{(i+1)!}\Big)\bar d_iv(n+m+k)\\
&\ +k\sum_{i=0}^r\sum_{j=0}^i\frac{m^{i+1-j}k^j}{(i+1-j)!j!}\bar
d_iv(n+m+k)\\
&\ -m\sum_{i=0}^r\sum_{j=0}^i\frac{k^{i+1-j}m^j}{(i+1-j)!j!}\bar
d_{i}v(n+m+k)\\
=&\ (k-m)(\a+n) v(n+m+k)\\
&\ +k\sum_{i=0}^r\sum_{j=0}^{i+1}\frac{m^{i+1-j}k^j}{(i+1-j)!j!}\bar
d_{i}v(n+m+k)\\ &\
-m\sum_{i=0}^r\sum_{j=0}^{i+1}\frac{k^{i+1-j}m^j}{(i+1-j)!j!}\bar
d_{i}v(n+m+k)\\ =&\ (k-m)\Big(\a+n
+\sum_{i=0}^r\frac{(m+k)^{i+1}}{(i+1)!}\bar d_{i}\Big)v(n+m+k).
\endaligned$$
Then $[d_m,d_k]v(n)=(d_md_k- d_kd_m)v(n)$. Consequently,  $\mathcal
{L}(M, \a)$ is a module over $\mathfrak{V}$.
\end{proof}

From (2.1), if $M$ is infinite dimensional, then $\mathcal {N}(M, \a
)=\oplus_{n\in\Z}\mathcal {N}_{\a+n}$ is a weight Virasoro module
with infinite dimensional weight spaces $\mathcal
{N}_{\a+n}=M\otimes t^n$ where
$$\mathcal {N}_{\a+n}=\{v\in \mathcal {N}(M,
\a)\,\,|\,\,d_0v=(\a+n)v\}.$$

\section{Simplicity and the isomorphism classes of weight Virasoro modules $\mathcal {N}(M, \a)$}

In this section we will determine the simplicity and the isomorphism
classes of $\mathcal {N}(M, \a)$.

\begin{lemma}\label{trivial}Let $M$ be a simple module over $\mathfrak{a}_r$. Then
 either $\bar d_rM=0$ or the action of $\bar d_r$ on $M$ is
bijective.\end{lemma}
\begin{proof}It is straightforward to check that $\bar d_rM$ and
${\rm ann}_M(\bar d_r)=\{v\in M|\bar d_r v=0\}$ are submodules of
$V$. Then the lemma follows from the simplicity of $M$.
\end{proof}

For any $l,m\in \Z$,  by induction on $s\in \Z_+$ we define the
following important operators
$$\omega_{l,m}^{(0)}=d_{l-m}d_{m}\in U(\Vir),$$
$$\omega_{l,m}^{(1)}=\omega_{l,m+1}^{(0)}-\omega_{l,m}^{(0)},\,\,{\rm and}$$
$$\omega_{l,m}^{(s)}=\omega_{l,m+1}^{(s-1)}-\omega_{l,m}^{(s-1)}.$$

Thus \begin{equation}\omega_{l,m}^{(s)}=\sum_{i=0}^s
{s\choose{i}}(-1)^{s-i} d_{l-m-i}d_{m+i}\in U(\Vir).\end{equation}

\begin{lemma}\label{key-compute} Let $r\ge 0$, $M$ be a module over
$\mathfrak{a}_r$. Then for any $m,l\in \Z$, $s\in \Z_+$ we have
$$\omega_{l,m}^{(2r+2)}v(n)=(2r+2)!(-1)^{r+1}\bar d_r^2v(n+l),\forall v\in M,n\in
\Z,\,\, {\rm and}$$
$$\omega_{l,m}^{(s)}\mathcal {N}(M, \a)=0, \forall s>2r+2.$$
\end{lemma}

\begin{proof} Let us fix  $l,n\in \Z$ and $v\in M$. We have
$$\aligned &\ \omega_{l,m}^{(0)}v(n)=d_{l-m}d_mv(n)\\=&
  \Big(\a+n+m+\sum_{i=0}^r\frac{(l-m)^{j+1}}{(j+1)!}\bar d_j\Big)
\Big(\a+n+\sum_{i=0}^r\frac{m^{i+1}}{(i+1)!}\bar
d_i\Big)v(n+l).\endaligned$$ Note that the righthand of the equation
is a polynomial of the variable $m$ and has degree $2r+2$, which can
be written as
$$ (-1)^{r+1} m^{2r+2} \bar
d_r^2v(n+l)+(\sum_{j=0}^{2r+1}m^j v_{j,0}(n+l)),$$ where $v_{j,0}\in
U(\mathfrak{a}_r)v$ is independent of $m$. By induction on $s$ we
can easily see that
$$\omega_{l,m}^{(s)}v(n)=(-1)^{r+1}(2r+2)(2r+1)...(2r+3-s)) m^{2r+2-s} \bar
d_r^2v(n+l)$$ $$+(\sum_{j=0}^{2r+1-s}m^j v_{j,s}(n+l)),\hskip 4cm $$
where $v_{j,s}\in U(\mathfrak{a}_r)v$ is independent of $m$. Now the
lemma follows from this formula.
\end{proof}

Let $M$ be a module over $\mathfrak{a}_r$ such that $\mathcal {N}(M,
\a)$ is simple. From (\ref{def}), we know that $M$ is a simple
module over $\mathfrak{a}_r$. From Example 1 and Lemma
\ref{trivial}, we may further assume that $r\ge 1$ and the action of
$\bar d_r$ on $M$ is bijective.

Now we can prove our main result in this section.

\begin{theorem} Let $r\ge 1$, $M$ be a simple module over
$\mathfrak{a}_r$ such that the action of $\bar d_r$ on $M$ is
injective. Then for any $\a\in \C$, the weight Virasoro module
$\mathcal {N}(M, \a)$ is simple.
\end{theorem}
\begin{proof} Suppose $W$ is a nonzero Virasoro submodule of $\mathcal
{L}(M, \a)$. It suffices to show that $W=\mathcal {N}(M, \a)$. Since
$\mathcal {N}(M, \a ) =\oplus_{n\in\Z}\left(M\otimes t^{n}\right)$
is a weight Virasoro module, then $W= \oplus_{n\in\Z}W_n\otimes
t^{n}$ where $W_n$ are subspaces of $M$. Let
$W^{(0)}=\cap_{n\in\Z}W_n$. We will prove that $W^{(0)}$ is a
nonzero $\mathfrak{a}_r$-submodule.

\

{\bf Claim 1:} $W^{(0)}$ is nonzero.

Since $W\ne 0$, we may assume that $W_n\ne0$ for some $n\in\Z$. Take
a nonzero $v\in W_n$. From Lemma \ref{key-compute}, we see that
$$\omega_{l,m}^{(2r+2)}v(n)=(2r+2)!(-1)^{r+1}\bar d_r^2v(n+l)\in W, \forall l\in \Z.$$ Hence $\bar d_r^2v\in W^{(0)}$,
which is nonzero since the action of $\bar d_r$ on $M$ is injective.
Claim 1 follows.

\

{\bf Claim 2:}  $W^{(0)}$ is an  $\mathfrak{a}_r$-submodule.

For any  $v\in W^{(0)}$, we have $v(k-m)\in W_{k-m}$ for any
$k,m\in\Z$. We know that
$$d_mv(k-m)=\Big(a+k-m+\sum_{i=0}^r\frac{m^{i+1}}{(i+1)!}\bar
d_i\Big)v(k)\in W_k\otimes t^k,$$ yielding
$$\sum_{i=0}^r\frac{m^{i+1}}{(i+1)!}\bar
d_i v \in W_k, \forall \,\,m\in\Z.$$ Hence $\bar d_iv\in W_k$ for
all $k\in \Z$ and $i=0,1,2,..., r$. Therefore $W^{(0)}$ is an
$\mathfrak{a}_r$-submodule.

\

 Since $M$ is a simple module
over $\mathfrak{a}_r$, we see that $W^{(0)}=M$, i.e., $W_n=M$ for
all $n\in\Z$. Consequently, $W=\mathcal {N}(M, \a)$ and $\mathcal
{N}(M, \a)$ is a simple Virasoro module
\end{proof}

\begin{theorem} Let $\a, \a'\in \C$, $r,r'\ge 1$, $M, M'$ be simple modules over
$\mathfrak{a}_r$, $\mathfrak{a}_{r'}$ respectively such that the
actions of $\bar d_r, \bar d_{r'}$ on $M, M'$ are injective
respectively. Then  the Virasoro modules $\mathcal {N}(M, \a)$ and $
\mathcal {N}(M', \a')$ are isomorphic if and only if $\a-\a'\in \Z$,
$r=r'$ and $M\cong M'$ as $\mathfrak{a}_r$ modules.
\end{theorem}

\begin{proof}The sufficiency of the condition is clear. Now suppose that
$\phi:\mathcal {N}(M, \a)\rightarrow \mathcal {N}(M', \a')$ is a Virasoro module isomorphism.
Comparing the set of weights, we have $\a+\Z=\a'+\Z$, that is,
$\a-\a'\in \Z$. Say $\a=\a'+n_0$. Denote
$$\phi(v(0))=\psi(v)(n_0),\,\,\forall\  v\in M,$$ where $\psi$ is a
bijective linear map from $M$ to $M'$. Then from
$\phi(\omega_{l,m}^{(s)}v(0))=\omega_{l,m}^{(s)}\phi(v(0))$ and
Lemma \ref{key-compute}, we have $r=r'$ and $$\phi(\bar d_r^2
v(l))=\bar d_r^2 \psi(v)(n_0+l),\forall l\in \Z, v\in M.$$ Recall
from Lemma \ref{trivial} that $\bar d_r$ is an isomorphism of $M$.
So  $$\phi(v(l))=\bar d_r^2\psi(\bar d_r^{-2} v)(n_0+l), \forall\
n\in \Z, v\in M,$$ where $\bar d_r^{-1}$ is the inverse mapping of
$\bar d_r$.

By taking $l=0$, we get $\bar d_r^2\psi(\bar d_r^{-2} v)=\psi (v)$
and
$$\phi(v(l))=\psi(v)(l+n_0), \forall l\in \Z,v\in M.$$
From $\phi(d_m v(-m))=d_m \phi(v(-m))$, we get
$$\psi(\Big(\a-m+\sum_{i=0}^r\frac{m^{i+1}}{(i+1)!}\bar
d_i\Big)v)(n_0)$$
$$=\Big(\a-m+\sum_{i=0}^r\frac{m^{i+1}}{(i+1)!}\bar
d_i\Big)\psi(v)(n_0),$$ yielding that

$$\sum_{i=0}^r\frac{m^{i+1}}{(i+1)!}(\psi(\bar d_i v)-\bar d_i
\psi(v)))=0.$$

Since $m$ is arbitrary, we have $\psi(\bar d_i v)=\bar d_i \psi(v)$
for $i=0,1,\ldots,r$. Thus we have proved that $\psi$ is an
$\mathfrak{a}_r$ module isomorphism, which completes the proof.
 \end{proof}

\section{Simple Virasoro modules $\mathcal {N}(M, \a)$ are new}

 Known irreducible weight Virasoro modules
with infinite dimensional weight spaces are the ones obtained by
some tensor product of an irreducible highest (or lowest) weight
module and an irreducible Virasoro module of the intermediate
series, see \cite{Zh}, and the ones defined in \cite{CM}. Let us
first recall these modules.

\

Let $U:=U(\V)$ be the universal enveloping algebra of the Virasoro
algebra $\V$. For any $\dot c, h\in \C$, let $I(\dot c,h)$ be the
left ideal of $U$ generated by the set $$
\bigl\{d_{i}\bigm|i>0\bigr\}\bigcup\bigl\{d_0-h\cdot 1, c-\dot
c\cdot 1\bigr\}. $$ The Verma module with highest weight $(\dot c,
h)$ for $\V $ is defined as the quotient  $\bar V(\dot
c,h):=U/I(\dot c,h)$. It is a highest weight module of $\V $ and has
a basis consisting of all vectors of the form $$
d_{-i_1}d_{-i_2}\cdots d_{-i_k}v_{h};\quad k\in{\N}\cup\{0\},
i_{j}\in\N, i_{k}\geq\cdots\geq i_2\geq i_1>0.
$$ Then we have the {\it irreducible highest weigh module} $ V(\dot c,h)=\bar V(\dot c,h)/J$
where $J$ is the maximal proper submodule of $\bar V(\dot c,h)$.

 For $a,b\in \C$,  the $\V $-modules $V_{a,b}$ has basis $\{v_{a +i}|i\in \Z\}$ with
trivial central actions and
$$  d_{i}v_{a+k}=(a+k+ib)v_{a+k+i}.$$ It is known
that $V_{a,b}\cong V_{a+1,b},$ for all $a,b\in\C$ and that $V_{a,0}
\cong V_{a,1}$ if $a\notin\Z$. It is also clear that $V_{0,0}$ has
$\C v_0$ as a submodule, we denote the quotient module $V_{0,0}/\C
v_0$ by $V'_{0,0}$. Dually, $V_{0,1}$ has $\C v_0$ as a quotient
module, and its corresponding submodule $\oplus_{i\neq 0}\C v_i$ is
isomorphic to $V'_{0,0}$. For convenience we simply write
$V'_{a,b}=V_{a,b}$ when $V_{a,b}$ is irreducible.

Let us recall a result from  \cite{Zh}.

\begin{theorem}If $\bar V(\dot c,h)$ is   reducible, and if $a$ is transcendental
over $\Q( \dot c,h, b)$ or $a$ is algebraic over $\Q( \dot c,h, b)$
with sufficiently large degree, then  $V(\dot c,h)\otimes V'_{a,b}$
is irreducible.\end{theorem}

Now let us recall the weight modules from Lemma 2.1 in \cite{CM}.
For any $h,b,\gamma,p\in \C$, the Virasoro module $E_h(b,\gamma,p)$
has a basis $\{T_i^k|i\in \Z, k\in \Z_+\}$ with the action given by
$$\aligned \ d_n
T_i^k=&T_{i+n}^{k+1}(1-e^{nh})+T_{i+n}^k[b+i+n(p+\gamma-(\gamma+k)e^{nh})]\\
&-e^{nh}\sum_{j=0}^{k-1}T_{i+n}^jn^{k-j+1}\left[\left(
  \begin{array}{c}
    k \\
    j-1 \\
  \end{array}
\right)+\gamma\left(
  \begin{array}{c}
    k \\
    j \\
  \end{array}
              \right)\right],\\ c E_h(b,&\gamma,p)=0.\endaligned$$ \
Some sufficient conditions for $E_h(b,\gamma,p)$ to be simple were
given in \cite{CM}.  Now we can compare the modules.

\begin{theorem} Let $r\ge 1$,  $\a\in \C$, and $M$ be a simple module over
$\mathfrak{a}_r$ such that the action of $\bar d_r$ on $M$ is
injective. Then the Virasoro module $\mathcal {N}(M, \a)$ is not
isomorphic either to   $V(\dot c,h)\otimes V'_{a,b}$ for any $a, b,
\dot c, h\in\C$ or to  $E_h(b,\gamma,p)$ for any $b,h,\gamma,p\in
\C$.
\end{theorem}

\begin{proof} Let $v_1$ be the highest weight vector of $V(\dot c,h)$,   $v_2$ be a nonzero weight vector
of $V'_{a,b}$. By Example 1, $V_{a,b}\cong\mathcal {N}(V, a)$, where
$V=\C v$ is a one dimensional over the Lie algebra $\mathfrak{a}_0$
and is given by  a scalar $b\in\C$ with the action $\bar d_0v=bv$.

From   the fact that $v_1$ is a highest weight vector and Lemma
\ref{key-compute}, we have

$$\omega_{l,m}^{(3)} (v_1\otimes v_2)= v_1\otimes \omega_{l,m}^{(3)}v_2=0, \forall\ m>0, l>m+3.$$
 Hence
$$\omega_{l,m}^{(2r+2)}(v_1\otimes v_2)=v_1\otimes \omega_{l,m}^{(2r+2)}v_2=0, \forall\  m>0, l>m+2r+3.$$
However by Lemma \ref{key-compute}, the action of
$\omega_{m,l}^{(2r+2)}$ is injective on $\mathcal {N}(M, \a)$. So we
have proved $\mathcal {N}(M, \a)$ is not isomorphic to any Virasoro
module $V(\dot c,h)\otimes V'_{a,b}$.

Next suppose that $\mathcal {N}(M, \a)\cong E_h(b,\gamma,p)$ for
some $h,b,\gamma,p$. Note that $E_h(b,\gamma,p)$ is not simple if
$e^h=1$ (See Page 160 in \cite{CM}). So we have $e^h\ne 1$. Denote
$$E_{\le k}=\span\{T_i^j\in E_h(b,\gamma,p)\mid  j\le k,i\in \Z\},\forall\ k\in \Z_+.$$

From
$$\aligned&\sum_{i=0}^s{s\choose{i}}(-1)^{s-i}(1-e^{(-m-i)h})(1-e^{(m+i)h})\\
=&\sum_{i=0}^s{s\choose{i}}(-1)^{s-i}(2-e^{(-m-i)h}-e^{(m+i)h})\\
=&-e^{-mh}\sum_{i=0}^s{s\choose{i}}(-1)^{s-i}e^{-ih}-e^{mh}\sum_{i=0}^s{s\choose{i}}(-1)^{s-i}e^{ih}\\
=&-(e^{mh}(e^h-1)^s+e^{-mh}(e^{-h}-1)^s)
\endaligned$$
 we have
$$\omega_{0,m}^{(s)} T_i^k\in
-(e^{mh}(e^h-1)^s+e^{-mh}(e^{-h}-1)^s)T_i^{k+2}+E_{\le k+1},$$ for
all $k,s\in \N,i\in \Z.$

From Lemma \ref{key-compute}, we see that
$e^{mh}(e^h-1)^s+e^{-mh}(e^{-h}-1)^s=0$ for all $s>2r+2, m\in \Z,$
 which implies that $e^h=1$, a contradiction. So we have proved the theorem.
\end{proof}

 \section{Examples}

From Theorem 4 we know that, to obtain new irreducible weight
Virasoro modules $\mathcal {N}(M, \a)$, it is enough to construct
infinite dimensional irreducible modules $M$ over $\mathfrak{a}_r$
for $r>0$ such that the action of $\bar d_r$ on $M$ is injective.
These modules $M$ were considered in \cite{MZ3}. Let us recall some
of such irreducible $\mathfrak{a}_r$-modules.

\

%

  \noindent{\bf Example 2.} For any $\a_1,\a_2, \lambda_1,\lambda_2 \in \C$ with $\lambda_1$ or $\lambda_2 $ nonzero,
  we have the simple $\mathfrak{a}_2$ module
$M=\C[x]$ with the action
$$\bar{d_0} f(x)=(x+\a_1)f(x), $$ $$\bar{d_i} f(x)=\lambda_if(x-i),
i=1,2, \forall f(x)\in \C[x].$$ Then  we have the induced new
irreducible weight Virasoro module
$\Gamma(\a_1,\a_2,\lambda_1,\lambda_2)=\C[x,t,t^{-1}]$ with the
action
$$c \cdot x^it^n=0$$
$$d_m \cdot x^it^n=\Big(\big(\a_2+n+m(x+\a_1)\big)x^i+\frac{m^2}{2} \lambda_1 (x-1)^i
+\frac{m^3}{6} \lambda_2(x-2)^i\Big)t^{m+n},$$ for all $m,n\in \Z,
i\in \Z_+$. \qed

\

  \noindent{\bf Example 3.} Fix $r\in\mathbb{N}$ with $r>2$. Choose a pair
$(S,\lambda)$, where $S\subset\{1,2,\dots,r\}$ and
$\lambda=(\lambda_i)_{i\in S}\in\mathbb{C}^{|S|}$, such that the
following conditions are satisfied:
\begin{enumerate}[(I)]
\item\label{cond1} $r\in S$ and $\lambda_r\neq 0$;
\item\label{cond2} for all $i,j\in S$ with $i\neq j$,   $i+j\in S$ implies $\lambda_{i+j}=0$;
\item\label{cond3}
for all $j\in \{1,2,\dots,r\}\setminus S$, we have $r-j\in S$.
\end{enumerate}
One example
is $S=\{\lceil\frac{k}{2}\rceil,\lceil\frac{k}{2}\rceil+1,\dots,r\}$
and any $\lambda$ with  $\lambda_r\neq 0$. Another example is
$S=\{2,4,5\}$ for $r=5$ and any $\lambda$ with  $\lambda_5\neq 0$.
Our final example in this case is $S=\{3,4,6,7,8\}$ for $r=8$ and
any $\lambda$ with $\lambda_8\neq 0$ and $\lambda_7=0$ (note that
here we have $3+4\in S$).

Denote by $Q_{\lambda}$ the $\mathfrak{a}_r$-module
$U(\mathfrak{a}_r)/I$, where $I$ is the left ideal generated by $
d_{i}-\lambda_{i}$, $i\in S$. It was proved in \cite{MZ3} the
$\mathfrak{a}_r$-module $Q_{\lambda}$ is simple. Then we have
irreducible Virasoro modules $\mathcal {N}(Q_{\lambda}, \a)$ for any
$\a\in\C$.\qed

 \

 \noindent{\bf Example 4.} All irreducible modules over $\mathfrak{a}_1$ were
classified in \cite{Bl}. For example, a class of such irreducible
$\mathfrak{a}_1$-modules are of the form $M_n=U(\mathfrak{a}_1)/I$
where $I$ is the left ideal  generated by $\bar d_0^n\bar d_1-1$ for
any fixed $n\in \N$ (see \cite{AP} for details). Then we have
irreducible Virasoro modules $\mathcal {N}(M_n, \a)$ for any
$\a\in\C$.\qed

 \

 \noindent{\bf Example 5.} All irreducible modules over $\mathfrak{a}_2$ were
classified in \cite{MZ3}. Let $\psi:\mathfrak{a}_2\to
\mathfrak{a}_1$ be the unique Lie algebra epimorphism which sends
\begin{displaymath}
\bar d_0\to \frac{1}{2}\bar d_0,\quad \bar d_1\to 0,\quad \bar
d_2\to \bar d_1.
\end{displaymath}
For any $\lambda\in \mathbb{C}$, mapping
\begin{displaymath}
\bar d_0\to \bar d_0,\quad \bar d_1\to \bar d_1,\quad \bar d_2\to
\lambda\bar d_1^2
\end{displaymath}
extends to an epimorphism $\varphi_{\lambda}:U(\mathfrak{a}_2)\to
U(\mathfrak{a}_1)$.

For any simple $\mathfrak{a}_1$-module $M$, we have two different
ways to make $M$ into an irreducible $\mathfrak{a}_2$-modules:
$xv=\psi(x)v$ for all $x\in \mathfrak{a}_2$ and $v\in M$; or
$xv=\varphi_{\lambda}(x)v$ for all $x\in \mathfrak{a}_2$ and $v\in
M$. These exhaust  all irreducible $\mathfrak{a}_2$-modules. Then we
have irreducible Virasoro modules $\mathcal {N}(M, \a)$ for any
$\a\in\C$. Modules in Example 2 are special cases.\qed

 \section{Deformation of the modules $\mathcal {N}(M, 0)$}

In this section we will use the irreducible Virasoro modules
$\mathcal {N}(M, 0)$  and ``the twisting technique" to construct new
irreducible non-weight modules over the Virasoro algebra
$\mathfrak{V}$ with trivial action of the center. Let us first
recall the definition of the twisted Heisenberg-Virasoro algebra.

\

 The twisted Heisenberg-Virasoro algebra $\LL$ is the universal
central extension of the Lie algebra $\{f(t)\frac{d}{dt}+g(t)| f,g
\in \C[t,t^{-1}]\}$ of differential operators of order at most one
on the Laurent polynomial algebra $\C[t,t^{-1}]$. More precisely,
the twisted Heisenbeg-Virasoro algebra ${\LL}$ is a Lie algebra over
$\C$ with the basis
$$\{d_n,t^n ,z_1, z_2,z_3 \mid n \in \Z\}$$
and subject to the Lie bracket
\begin{equation}[d_n,d_m]=(m-n)d_{n+m}+\delta_{n,-m}\frac{n^3-n}{12}z_1,\end{equation}
\begin{equation}[d_n,t^m]=mt^{m+n}+\delta_{n,-m}(n^2+n)z_2, \end{equation}
\begin{equation}[t^n,t^m]=n\delta_{n,-m}z_3,\end{equation}
\begin{equation}[\LL,z_1]=[\LL,z_2]=[\LL,z_3]=0. \end{equation}

The Lie algebra $\LL$ has a Virasoro subalgebra $\mathfrak{V}$ with
basis $\{d_i,z_1|i\in \Z\}$, and a Heisenberg subalgebra
$\tilde{\mathcal {H}}$ with basis $\{t^i, z_3\mid i\in \Z\}$.

\

Let us start with  a  Virasoro module $\mathcal {N}(M, 0)$ defined
in (2.1) and (2.2) with $\a=0$. On $\mathcal {N}(M, 0)$ We can
extend the Virasoro structure to a module structure over $\LL$ as
follows
$$t^k (v\otimes t^n)= v\otimes t^{n+k}, \forall\  k,n\in\Z,v\in M,$$
$$z_1 \mathcal {N}(M, 0)=0, z_2 \mathcal {N}(M, 0)=0, z_3 \mathcal {N}(M, 0)=0.$$

 For any $\b\in \C[t, t^{-1}]$, by using the twisting technique (see (2.3) in \cite{LGZ})
 we define the new action of $\mathfrak{V}$ on $\mathcal {N}(M, 0)$ by

\begin{equation}d_n \circ v=(d_n +\beta t^n) v,\forall \,\, n\in \Z,
v\in \mathcal {N}(M, 0),\end{equation}

\begin{equation}c\circ \mathcal {N}=0.\end{equation}
 We denote by $\mathcal {N}(M, \beta)$ the resulting module over $\mathfrak{V}$.

If $\beta\in\C$, we can easily see that $\mathcal {N}(M, \beta)$ is
exactly the modules define in Sect.2. If $M$ is one dimensional, it
is not hard to see that $\mathcal {N}(M, \beta)=\C[t,t^{-1}]$ with
actions
$$d_m t^n=(\b+n+bm)t^{n+m},\forall m,n\in\Z
$$ for some $b\in\C$.
These are exactly the modules studied in \cite{LGZ}. Recall that
this module is reducible if and only if $b=0$ and $\a\in \Z$, or $
b=1$ and $\a\in \Z\cup(\C[t^{\pm}]\setminus\C)$, see \cite{LGZ} for
the details.

\

Now we can prove the irreducibility of the modules $\mathcal {N}(M,
\beta)$.

\

\begin{theorem} Let $r\ge 1$, $M$ be a simple module over
$\mathfrak{a}_r$ such that the action of $\bar d_r$ on $M$ is
injective. Then for any $\beta\in \C[t, t^{-1}]\setminus \C$, the
Virasoro module $\mathcal {N}(M, \beta)$ is simple.
\end{theorem}
\begin{proof} Suppose that $W\ne 0$ is a Virasoro submodule of $\mathcal {N}(M,
\beta)$. It suffices to show that $W=\mathcal {N}(M, \beta)$. Let us
fix a nonzero $ w=\sum_{i\in \Z}v_i\otimes t^i\in W$, where the sum
is finite. We will use the element $\omega_{l,m}^{(0)}$ defined in
(3.1). For any $l, m\in \Z$, the vector $\omega_{l,m}^{(0)}\circ w$
can be written as
$$\aligned\omega_{l,m}^{(0)}\circ w=&\ d_{l-m}\circ d_m\circ w\\
=&\ (d_{l-m}+\beta t^{l-m})(d_m+\beta t^m) w\\
=&\ \omega_{l,m}^{(0)} w+d_{l-m}\beta t^mw+\beta
t^{l-m}d_mw+\beta^2t^lw\\
=&\ \omega_{l,m}^{(0)} w+ \sum_{i=0}^{r+2} m^j w_{j,0},\endaligned$$
where $w_{j,0}\in \mathcal {N}(M, 0)$ is independent of $m$. By
induction on $s$ we can prove that
$$\omega_{l,m}^{(s)}\circ w=\omega_{l,m}^{(s)} w+ \sum_{i=0}^{r+2-s} m^j
w_{j,s}$$ where $w_{j,s}\in \mathcal {N}(M, 0)$ is independent of
$m$. Using Lemma 3 we see that
\begin{equation}\omega_{l,m}^{(2r+2)}\circ
w=\omega_{l,m}^{(2r+2)}w=(2r+2)!(-1)^{r+1}\sum_i\bar d_r^2
v_{i}\otimes t^{i+l}\in W,\end{equation}
 for all $l,  m\in \Z$.
 In particular,   $W'=\{w\in W| t^i w\in W, \forall i\in \Z\}$
 is nonzero. Now for any $v\in W'$, we have
 $d_i v=d_i\circ v-t^i\beta v\in W$, and $$t^j(d_i v)=-j t^{i+j}v+d_i\circ
 (t^j v)-t^{i+j}\beta v\in W,\,\,\,\forall\  i,j\in \Z.$$
 So $d_iv\in W'$ for all $v\in W'$ and $i\in\Z$. Hence $W'$ is a nonzero submodule
 of the simple Virasoro module $\mathcal {N}(M, 0)$, which has to be
  $\mathcal {N}(M, 0)=\mathcal {N}(M, \beta)$. Therefore $W'=\mathcal {N}(M,
  \beta)$, and consequently, $W=\mathcal {N}(M, \beta)$.
\end{proof}

We like to conclude  this paper by comparing the nonweight
irreducible Virasoro modules $\mathcal {N}(M, \beta)$ with other
known irreducible Virasoro modules. When $\dim M=1$, $\mathcal
{N}(M, \beta)$ are exactly the modules studied in \cite{LGZ}. The
following result handles the remaining cases.

\begin{theorem}\label{Comp} Let $r\ge 1$,  $\beta\in \C[t,t^{-1}]\setminus \C$,
and $M$ be a simple module over $\mathfrak{a}_r$ such that the
action of $\bar d_r$ on $M$ is injective. Then $\mathcal {N}(M,
\beta)$ is a new irreducible Virasoro module.\end{theorem}

\begin{proof} All other known non-weight Virasoro modules are from
\cite{LZ, MZ3, MW}. Let us recall those modules defined in
\cite{MZ3}. Let $\mathfrak{V}_+$ $=\span\{d_i\,|\,i\in\Z_+\}$. Given
$N\in {\mathfrak{V}}_+\text{-}\mathrm{mod}$ and
$\theta\in\mathbb{C}$, consider the corresponding induced module
$\mathrm{Ind}(N):=U(\mathfrak{V})\otimes_{U(\mathfrak{V}_+)}N$ and
denote by $\mathrm{Ind}_{\theta}(N)$ the module
$\mathrm{Ind}(N)/(\mathtt{c}-\theta)\mathrm{Ind}(N)$. It was proved
that $V=\mathrm{Ind}_{\theta}(N)$ is irreducible over $\mathfrak{V}$
if $d_kN=0$ for all sufficiently large $k$. For any $v\in V$ we have
$d_kv=0$ for all sufficiently large $k$. Now we use the operator
$\omega_{l,m}^{(2r+2)}$ defined in (3.1). From (6.7) we know that
$\omega_{l,m}^{(2r+2)}\circ w\ne0$ for any nonzero $w\in \mathcal
{N}(M, \beta)$ and any $l,m\in\Z$, while for any nonzero $v\in V$ we
have $\omega_{l,m}^{(2r+2)}v=0$ if  $l>2m>>0$. Thus $\mathcal {N}(M,
\beta)$ is not isomorphic to $V$.

\

Now we compare our module $\mathcal {N}(M, \beta)$ with the
irreducible Virasoro modules $W=$Ind$_{\theta, z}(\C_{\bf m})$
defined in \cite{MW}, where $\theta, m_2,m_3,$ $ m_4\in\C$ and
$z\in\C^*$ satisfying the conditions \begin{equation}zm_3\ne m_4,
2zm_2\ne m_3,3zm_3\ne 2m_4,z^2m_2+m_4 \ne2z m_3.\end{equation} There
is a nonzero vector $v\in W$ such that
$$(d_2-zd_1)v=m_2v,\,\,\,(d_3-z^2d_1)v=m_3v,\,\,\,(d_4-z^3d_1)v=m_3v.$$
And $$(d_i-z^{i-1}d_1)v=-(i-4)m_3z^{i-3}v+(i-3)m_4z^{i-4}v,\ \forall
\ i>4.$$

Then
$$\aligned &\left(d_{i+1}-2zd_{i}+z^2d_{i-1}\right)v\\
=&-(i-3)m_3z^{i-2}v+(i-2)m_4z^{i-3}v\\
& +2(i-4)m_3z^{i-2}v-2(i-3)m_4z^{i-3}v\\
& -(i-5)m_3z^{i-2}v+(i-4)m_4z^{i-3}v\\
=& \ 0,\  \forall\ i\ge4.
\endaligned$$
By (3.1), we see that $$\aligned &(\omega_{5,5}^{(2r+2)}-2z
\omega_{4,4}^{(2r+2)}+z^2\omega_{3,3}^{(2r+2)})v\\
=&\sum_{i=0}^{2r+2}
    {2r+2\choose{i}}
                (-1)^{2r+2-i}d_{-i}(
d_{5+i}-2zd_{4+i}+z^2d_{3+i})v=0.\endaligned$$ While we know from
(6.7) that $(\omega_{5,5}^{(2r+2)}-2z
\omega_{4,4}^{(2r+2)}+z^2\omega_{3,3}^{(2r+2)})\circ w\ne0$ for any
nonzero $w\in \mathcal {N}(M, \beta)$. So $\mathcal {N}(M, \beta)$
is not isomorphic to $W$.

\

At last we compare our module $\mathcal {N}(M, \beta)$ with the
irreducible Virasoro modules $A_b$ defined in \cite{LZ} where
$b\in\C\setminus\{1\}$ and $A_b$ is an irreducible module over the
associative algebra $\C[t, t^{-1},\frac d{dt}]$ such that $\C[\frac
d{dt}]$ is torsion-free on $A_b$ (otherwise $A_b$ is a weight
module). Recall that $\partial=t\frac d{dt}$. From (3.5) in
\cite{LZ} we deduce that, for any $v\in A_b$,
$$\aligned\omega_{l,m}^{(0)}\circ v=&\ d_{l-m}\circ d_m\circ v\\
=&\ \big(t^{l-m}\partial+(l-m)bt^{l-m}\big)(t^{m}\partial+mbt^{m})v\\
=&\ t^{l}(\partial^2+(m+lb)\partial+(m(1-b)+lb)mb)v.
\endaligned$$Then we can compute

$$\aligned &\omega_{l,m}^{(s)}\circ v=\sum_{i=0}^s {s\choose{i}}(-1)^{s-i}
d_{l-m-i}\circ d_{m+i}\circ v\\
=&\ t^{l}\sum_{i=0}^s
{s\choose{i}}(-1)^{s-i}\left(\partial^2+(m+i+lb)\partial+((m+i)(1-b)+lb)(m+i)b\right)v\\
=&\ t^{l}\sum_{i=0}^s
{s\choose{i}}(-1)^{s-i}\left(i\partial+i^2(1-b)b+b(2m(1-b)
+lb)i\right)v\\
=& \ 0,\ \forall\ l,m\in\Z, s\ge 3, v\in A_b.
\endaligned$$

 From (6.7) we know that  our module $\mathcal
{N}(M, \beta)$ does not satisfy this property. Thus $\mathcal {N}(M,
\beta)$ is not isomorphic to any $A_b$. This completes the proof.
\end{proof}



\vspace{1cm}

\noindent  G.L.: College of Mathematics and Information Science,
Henan University, Kaifeng, Henan  475004, China. Email:
liugenqiang@amss.ac.cn

\vspace{0.2cm} \noindent  R.L.: Department of Mathematics, Suzhou
university, Suzhou 215006, Jiangsu, P. R. China.
 Email: rencail@amss.ac.cn

\vspace{0.2cm} \noindent K.Z.: Department of Mathematics, Wilfrid
Laurier University, Waterloo, ON, Canada N2L 3C5,  and College of
Mathematics and Information Science, Hebei Normal (Teachers)
University, Shijiazhuang, Hebei, 050016 P. R. China. Email:
kzhao@wlu.ca

\end{document}